\documentclass[leqno,11pt]{amsart}
\usepackage{amsmath,amsthm,amsfonts,amssymb,amstext}
\usepackage{graphicx}
\usepackage{hyperref}
\usepackage{setspace}
\usepackage{color}
\usepackage[margin=1.1in]{geometry}

\numberwithin{equation}{section}
\theoremstyle{plain}

\newtheorem{theorem}{Theorem}
\newtheorem*{theorem*}{Theorem}

\newtheorem{corollary}[theorem]{Corollary}

\newtheorem{lemma}[theorem]{Lemma}

\newtheorem{remark}[theorem]{Remark}

\newtheorem*{proposition*}{Proposition}
\newtheorem*{claim*}{Claim}
\newtheorem*{lemma*}{Lemma}

\newcommand{\E}{\mathbb{E}}
\newcommand{\PR}{\mathbb{P}}
\newcommand{\eps}{\epsilon}
\newcommand{\poi}{\mathcal{P}}

\newcommand{\R}{\mathbb{R}}
\newcommand{\Z}{\mathbb{Z}}
\newcommand{\I}{\mathbb{I}}

\newcommand{\del}{\delta}
\newcommand{\Del}{\Delta}

\begin{document}

\title[The largest component in the Bohman-Frieze process]{On the largest component in the subcritical regime of the Bohman-Frieze process}

\author{Sanchayan Sen}
\address
{Courant Institute of Mathematical Sciences\newline
\indent New York University\newline
\indent 251 Mercer Street\newline
\indent New York, NY-10012\newline
\indent United States of America}
\email{
sen@cims.nyu.edu}

\date{4 April 2013. \textit{Revised:} 9 May 2013}

\begin{abstract}
Kang, Perkins and Spencer \cite{kang} showed that the size of the largest component
of the Bohman-Frieze process at a fixed time $t$ smaller than $t_c$,
the critical time for the process is $L_1(t)=\Omega(\log n/(t_c-t)^2)$ with high probability.
They also conjectured that
this is the correct order, that is $L_1(t)=O(\log n/(t_c-t)^2)$ with high probability for fixed $t$ smaller
than $t_c$. Using a different
approach, Bhamidi, Budhiraja and Wang \cite{bhamidi} showed that $L_1(t_n)=O((\log n)^4/(t_c-t_n)^2)$ with
high probability
for $t_n\leq t_c-n^{-\gamma}$ where $\gamma\in(0,1/4)$. In this paper, we improve the result in
\cite{bhamidi} by showing that
for any fixed $\lambda>0$, $L_1(t_n)=O(\log n/(t_c-t_n)^2)$
with high probability for $t_n\leq t_c-\lambda n^{-1/3}$.
In particular, this settles the conjecture in \cite{kang}. We also prove some generalizations
for general bounded size rules.
\end{abstract}

\maketitle

\section{Introduction}\label{sec:intro}
Initiated by a question of Dimitris Achlioptas, the study of
modified Erd\H{o}s--R\'{e}nyi processes (called Achlioptas processes)
has grown into a large area of research in the past decade.
At each step of an Achlioptas process, two randomly chosen edges are presented and
one of these edges is added to the current random graph according to some selection rule.
The behavior of the random graph process (e.g. point of phase transition, appearance
of Hamiltonian cycles) depends on the selection rule. The Erd\H{o}s--R\'{e}nyi
process, for example, is an Achlioptas process where the first edge chosen is always
added to the random graph. An account of the literature on properties of several Achlioptas
processes can be found in
\cite{3beveridge, bhamidi, bhamidi2, bhamidi3, 6bohman, 7bohman, 9bohman, janson, spencer, kang, 25KLS, riordan}
and the references therein.

One such Achlioptas process, called the Bohman-Frieze process has received a great deal of attention
and will be of particular interest to us.
We first describe a continuous time version of the Bohman-Frieze process.
Consider the complete graph $G_n=(V_n,E_n)$ on the vertex set $V_n=\{1,\hdots,n\}$.
Consider independent Poisson processes $\poi_{e}$ indexed by $e=(e_1,e_2)\in E_n\times E_n$
each having rate $2/n^3$. Let $\cup_{e\in E_n\times E_n}\poi_{e}=\{u_1<u_2<\hdots\}$.
Then the dynamics of the continuous time Bohman-Frieze process ($BF_n(\cdot)$) are given as
follows.\\
\phantom{m}\hskip15pt $BF_n(u)$  is the empty graph on $V_n$  for $0\leq u<u_1$.\\
\phantom{m}\hskip15pt If the Poisson process $\poi_{e}$ has a point at $u_i$ where $e=(e_1,e_2)$ and
the endpoints of $e_1$\\
\phantom{m}\hskip15pt are isolated vertices in $BF_n(u_i-)$,
set $BF_n(u)=BF_n(u_i-)\cup e_1\text{ for }u\in[u_i,u_{i+1})$.\\
\phantom{m}\hskip15pt Otherwise, set $BF_n(u)=BF_n(u_i-)\cup e_2\text{ for }u\in[u_i,u_{i+1})$.

The expected number of edges at $u=1$ is
$$\frac{2}{n^3}\times\dbinom{n}{2}^2\approx\frac{n}{2},$$
thus the time normalization is the one corresponding to the Erd\H{o}s--R\'{e}nyi process.

The corresponding discrete time version ($DBF_n(\cdot)$) of the Bohman-Frieze process
evolves as follows.\\
\phantom{m}\hskip15pt $DBF_n(u)$  is the empty graph on $V_n$  for $0\leq u<2/n$.\\
\phantom{m}\hskip15pt At time $2(k+1)/n$, two edges $e_1$ and $e_2$ are selected uniformly (with replacement)\\
\phantom{m}\hskip15pt from $E_n$. If the endpoints of $e_1$ are isolated vertices in $DBF_n(2k/n)$, set\\
\phantom{m}\hskip15pt $DBF_n(u)=DBF_n(2k/n)\cup e_1\text{ for }u\in[2(k+1)/n,2(k+2)/n)$.\\
\phantom{m}\hskip15pt Otherwise, set $DBF_n(u)=DBF_n(2k/n)\cup e_2\text{ for }u\in[2(k+1)/n,2(k+2)/n)$.

We shall denote by $L^{BF}_1(t)$ (resp. $L^{DBF}_1(t)$), the size of the largest component
of $BF_n(t)$ (resp. $DBF_n(t)$). It is known that the phase transition for the discrete time Bohman-Frieze process
happens at time $t_c>1$. (It is easy to see that $t_c$ is also the critical time for the continuous time Bohman-Frieze process
and hence we shall refer to it as the critical time for the Bohman-Frieze process.) Theorem $4$ of \cite{kang} shows that for
any fixed $t\in(0,t_c)$,
$$\PR(L^{DBF}_1(t)\geq K\log n/(t_c-t)^2)\to 1$$
for some constant $K$ free of $t$. (Actually, the version of the Bohnman-Frieze process
studied in \cite{kang} is slightly different from $DBF_n(\cdot)$ since we are sampling the edges
at time $2(k+1)/n$ from $E_n$ whereas, in \cite{kang}, only the edges not present in the
graph at time $2k/n$ are allowed, but this difference is negligible, for details see \cite{spencer}.)
Kang, Perkins and Spencer conjecture that this is indeed the correct order, i.e.
$$\PR(L^{DBF}_1(t)\leq K'\log n/(t_c-t)^2)\to 1 $$
for any fixed $t\in(0,t_c)$ and some constant $K'$ free of $t$ (Conjecture $1$ in \cite{kang}).
Bhamidi, Budhiraja and Wang \cite{bhamidi, bhamidi2} independently show that for $\gamma\in(0,1/4)$,
there exists a constant $C=C(\gamma)$ such that
\begin{equation}\label{eqn:bhamidi}
\PR(L^{BF}_1(t_n)\geq C(\log n)^4/(t_c-t_n)^2)\to 0\text{ for }t_n\leq t_c-n^{-\gamma}
\end{equation}
by connecting the dynamics of $BF_n(\cdot)$ to an inhomogeneous random graph model.
In this work, we take the approach in \cite{bhamidi}
and go through a more careful analysis to prove Conjecture $1$ of \cite{kang} (Theorem \ref{thm:continuous time BF}
and Corollary \ref{cor:discrete time BF}). Our result is true for $t_n\leq t_c-\lambda n^{-1/3}$
(for any fixed $\lambda>0$) and thus closes the gap between the critical window
and the interval $0<t\leq t_c-n^{-\gamma}, \gamma<1/4$ where the bound in \cite{bhamidi} is valid.

Spencer and Wormald \cite{spencer}  introduced a generalization of the Bohman-Frieze process
called bounded size rules. For $K\geq 0$, we define
$$\Omega_K=\{1,\hdots,K,\omega\}.$$
The symbol $\omega$ represents ``numbers bigger than $K$'' (see \cite{spencer}).
For a graph $G$ on $V_n$ and $v\in V_n$, let $\mathfrak{C}(v,G)$ denote the size of the component
of $G$ containing $v$. Let
\begin{align}
c(v,G)=
\left\{
\begin{array}{l}
\mathfrak{C}(v,G), \text{ if }\mathfrak{C}(v,G)\leq K,\\
\omega,\text{ otherwise}.
\end{array}
\right.\nonumber
\end{align}
Fix $F\subset\Omega_K^4$.
Consider independent Poisson processes $\{\poi_e: e\in V_n^4\}$ of intensity $1/2n^3$ and let
$\cup_{e\in V_n^4}\poi_{e}=\{u_1<u_2<\hdots\}$.
Then the continuous time bounded size rule process ($BSR(\cdot)$)
associated with $F$ can be described in the following way.\\
\phantom{m}\hskip15pt Define $BSR_n(u)$  to be the empty graph on $V_n$  for $0\leq u<u_1$.\\
\phantom{m}\hskip15pt If the Poisson process $\poi_{e}$ has a point at $u_i$ and $e=(v_1,v_2,v_3,v_4)$,
let $c(e,BSR_n(u_i-)):=$\\
\phantom{m}\hskip15pt $(c(v_1,BSR_n(u_i-)),\hdots,c(v_4,BSR_n(u_i-)))$.\\
\phantom{m}\hskip15pt If $c(e,BSR_n(u_i-))\in F$, set
$BSR_n(u)=BSR_n(u_i-)\cup \{v_1,v_2\}\text{ for }u\in[u_i,u_{i+1})$.\\
\phantom{m}\hskip15pt Otherwise set $BSR_n(u)=BSR_n(u_i-)\cup \{v_3,v_4\}\text{ for }u\in[u_i,u_{i+1})$.

The discrete time version, $DBSR_n(\cdot)$ can be defined in the same way we defined
$DBF_n(\cdot)$. In \cite{bhamidi2}, it was shown that \eqref{eqn:bhamidi}
holds if we replace $L_1^{BF}(t_n)$ by $L_1^{BSR}(t_n)$, the size of
the largest component in $BSR(t_n)$. We improve this result in Theorem \ref{thm:continuous time BSR}
(see also Corollary \ref{cor:discrete time BSR}).

The precise statements of our results are given in Section \ref{sec:statements}.
The proofs for the Bohman-Frieze process and the bounded size rules are given in
Section \ref{sec:Proofs} and Section \ref{sec:proofs1} respectively.
Although bounded size rules generalize the Bohman-Frieze rule,
we choose to treat the more important Bohman-Frieze process separately
in Section \ref{sec:Proofs} to retain clarity in the (notationally simpler)
proof. In Section \ref{sec:proofs1}, we give an outline
of the proof for general bounded size rules highlighting only those parts of the
proof which are slightly different.

\section{Main Results}\label{sec:statements}
The following theorem is our main result.
\begin{theorem}\label{thm:continuous time BF} Let $t_c$ be the critical time for the Bohman-Frieze process.
Let $t=t(n)$ satisfy
$t\leq t_c-\lambda n^{-1/3}$ for a fixed $\lambda>0$. Let $L_1^{BF}(t)$ denote the size of the largest component
of $BF_n(t)$. Then there exists a positive constant $\overline{C}$ depending only on $\lambda$ such that
$$\PR(L_1^{BF}(t)>\overline{C}\log n/(t_c-t)^2)\stackrel{n\to\infty}{\longrightarrow}0.$$
\end{theorem}
An immediate consequence of Theorem \ref{thm:continuous time BF} is the analogous
result for the discrete time Bohman-Frieze process.
\begin{corollary}\label{cor:discrete time BF} Let $t_c$ be the critical time for the Bohman-Frieze process.
Let $t=t(n)$ satisfy
$t\leq t_c-\lambda n^{-1/3}$ for a fixed $\lambda>0$. Let $L_1^{DBF}(t)$ denote the size of the largest component
of $DBF_n(t)$. Then there exists a positive constant $\overline{\overline{C}}$ depending only on $\lambda$ such that
$$\PR(L_1^{DBF}(t)>\overline{\overline{C}}\log n/(t_c-t)^2)\stackrel{n\to\infty}{\longrightarrow}0.$$
\end{corollary}
The arguments used in the proof of Theorem \ref{thm:continuous time BF}
can be generalized to general bounded size rules.
Recall the definition of $\Omega_K$ from Section \ref{sec:intro}.
\begin{theorem}\label{thm:continuous time BSR}
Fix $K\geq 2$, $F\subset\Omega_K^4$ and consider the bounded size rule associated with $F$.
Let $t_c$ be the critical time for the process. Let $L_1^{BSR}(u)$ denote the size of the 
largest component of $BSR_n(u)$ for $u\geq 0$. Then, there exists
$\zeta=\zeta(F)>1/4$ having the following property: for every fixed $\zeta'\in(0,\zeta)$, there exists a constant
$\overline{C}$ depending only on $\zeta'$ and $F$ such that
$$\PR(L_1^{BSR}(t)>\overline{C}\log n/(t_c-t)^2)\stackrel{n\to\infty}{\longrightarrow}0$$
whenever $t=t(n)$ satisfies $t\leq t_c-n^{-\zeta'}$.
\end{theorem}
An analogous result is true for the discrete time process.
\begin{corollary}\label{cor:discrete time BSR}
Let $F$, $\zeta$ be as in Theorem \ref{thm:continuous time BSR}. 
Let $L_1^{DBSR}(u)$ denote the size of the
largest component of $DBSR_n(u)$ for $u\geq 0$. 
Let $t_c$ be the critical time for the process.
Then, for every fixed $\zeta'\in(0,\zeta)$, there exists a constant
$\overline{\overline{C}}$ depending only on $\zeta'$ and $F$ such that
$$\PR(L_1^{DBSR}(t)>\overline{\overline{C}}\log n/(t_c-t)^2)\stackrel{n\to\infty}{\longrightarrow}0$$
whenever $t=t(n)$ satisfies $t\leq t_c-n^{-\zeta'}$.
\end{corollary}
\begin{remark}
It will follow from our arguments that
the conclusions of Theorem \ref{thm:continuous time BF} and Corollary \ref{cor:discrete time BF}
remain true for any bounded size rule with $K=1$.

We shall see in Section \ref{sec:examples} that if $K=2$ and $F$ satisfies
some conditions, then the constant $\zeta(F)$  is at least $1/3$,
hence the upper bounds in Theorem \ref{thm:continuous time BSR} and Corollary \ref{cor:discrete time BSR}
hold up to the critical window.

However, if $K\geq 3$ then the constant $\zeta$ that we get from our proof
will be smaller than $1/3$. So, the interval where the stated upper bound
holds falls shy of the critical window.
\end{remark}
\section{Proofs for the Bohman-Frieze process}\label{sec:Proofs}
We shall present the proofs of Theorem \ref{thm:continuous time BF} and
Corollary \ref{cor:discrete time BF} in this section.  We start off with
\subsection{Connection between the Bohman-Frieze process and an Inhomogeneous Random Graph Model}\label{sec:connection}
The analysis in \cite{bhamidi} is carried out by relating the dynamics of the Bohman-Frieze (BF)
process to an inhomogeneous random graph model. Since this plays a crucial role in the proof,
we briefly describe this connection and introduce the necessary notations. A more detailed
discussion of these and related models can be found in \cite{aldous}, \cite{bhamidi}
and \cite{bollobas}.

Let $X_n(v)$ be the number of singletons (i.e. isolated vertices) in the BF process at time $v$ and set
$x_n(v)=X_n(v)/n$. An edge added in the BF process at time $v$ can be of the following types:\\
(I) both its endpoints were isolated vertices in $BF_n(v-)$;\\
(II) only one of its endpoints was an isolated vertex in $BF_n(v-)$;\\
(III) none of its endpoints were isolated in $BF_n(v-)$.

Two singletons are added in the BF process (i.e. an edge of type I is created) if
one of the following happens:\\
(i) the first edge selected connects two isolated vertices or \\
(ii) the first edge selected does
not connect two isolated vertices but the second edge selected joins two isolated vertices.

Hence two singletons are added in the BF process at a rate
$$\frac{2}{n^3}
\left[\dbinom{X_n(v)}{2}\dbinom{n}{2}+\left(\dbinom{n}{2}-\dbinom{X_n(v)}{2}\right)\dbinom{X_n(v)}{2}\right]
=:na_n(x_n(v))$$
where $a_n(y)=a_0(y)+O(1/n)$ and $a_0:[0,1]\to\R_+$ is the function $a_0(y)=y^2-y^4/2$.

One can similarly show that a given non-singleton vertex (i.e. a vertex which is not isolated)
in $BF_n(v)$ gets connected
to some isolated vertex (i.e. an edge of type II is created) at a rate $c_n(x_n(v))$
where $c_n(y)=c_0(y)+O(1/n)$ and the function $c_0:[0,1]\to\R_+$ is given by $c_0(y)=(1-y^2)y$.

Finally, an analogous computation will show that two given non-singleton vertices
are joined (i.e. an edge of type III is created)
at a rate $b_n(x_n(v))/n$  where
$b_n(y)=b_0(y)+O(1/n)$ and the function $b_0:[0,1]\to\R_+$ is given by
$b_0(y)=1-y^2$.

The function $x_n(v)$ is highly concentrated around $x(v)$ (see \cite{spencer, bhamidi}),
which satisfies the ODE
$$x'(v)=-x(v)-x^2(v)+x^3(v),\ x(0)=1.$$
As a result, the rate functions, say for example $a_n(x_n(v))$ lies very close
to $a_0(x(v))$ with high probability. Hence, the BF process can be approximated
with high probability by a random graph model with deterministic rate functions
which we describe next.

Let $T=2t_c$ and let $a,b,c:[0,T]\to[0,1]$ be continuous. Consider the random graph
process with immigrating vertices and attachment (RGIVA) denoted by $IA_n(.)=IA_n(a,b,c)$
which can be described as follows:\\
(I) $IA_n(0)$  is the empty graph.\\
(II) Given $IA_n(v)$ for some $v\in[0,T)$,
a new component consisting of two vertices and an edge connecting them
(which we call a doubleton) is born in the interval $(v,v+dv]$
with rate $na(v)$;\\
(III) For any vertex $i$ in $IA_n(v)$,  a new vertex is born
and attaches itself to $i$  with an edge in $(v,v+dv]$ with rate $c(v)$;\\
(IV) Given two vertices $i_1,i_2$  in $IA_n(v)$, a new edge connecting them
is added in $(v,v+dv]$ with rate $b(v)/n$.\\
The last three events happen independently.\\
This model (introduced in \cite{bhamidi}) is a generalization of the RGIV model
of \cite{aldous} and is a good approximation of the BF process when
$a(v)=a_0(x(v))$, $b(v)=b_0(x(v))$ and $c(v)=c_0(x(v))$.

For fixed $v$, $IA_n(a,b,c)_v$ can be obtained by the following two step construction:

(Step 1) Run a Poisson process $\poi$ with rate $na(s)ds$ in the interval $[0,v]$.
At each arrival of this Poisson process a new doubleton is born. A doubleton born at time
$s\in[0,v)$ grows its own component independent of other components according to the following
Markov process: if $w(.)$ denotes the size of the component of the doubleton born at time $s$
as a function of time,
then given $\{w(u)\}_{s\leq u\leq s_1}$, the component size grows by one in $(s_1,s_1+ds_1]$
with rate $w(s_1)c(s_1)$. Continue this process till time $v$. In the end we shall
have a collection of pairs $\{(s_i,w_i)\}_{s_i\in\poi}$ (called clusters) where
$w_i:[0,v]\to\mathbb{\Z}_{\geq 0}$ is the function such that $w_i(u)$ denotes the size at time $u$, of the component
of the doubleton born at time $s_i$ with the convention $w_i(u)=0$ for
$u<s_i$.

(Step 2) Conditional on step 1, add an edge between two clusters $x=(s_1,w_1)$, $y=(s_2,w_2)$
with probability
$p_{n,v}(x,y):=1-\exp[-\frac{1}{n}\int_0^v w_1(u)w_2(u)b(u) du]$ for every
$x,y\in \{(s_i,w_i)\}_{s_i\in\poi}$ independently.

The distribution of the component sizes of the graph obtained via the two step process
is the same as that of $IA_n(a,b,c)_v$. This allows one to view $IA_n(a,b,c)$ as
an inhomogeneous random graph (IRG) model whose construction we recall next.

A triplet $(\overline{X},\overline{T},\overline{\mu})$ where $\overline{X}$
is Polish, $\overline{T}$ is the Borel $\sigma$-field and $\overline{\mu}$
is a finite measure is called a type space.

A measurable function $\overline{k}:\overline{X}\times \overline{X}\to[0,\infty)$
is called a kernel. Each kernel $\overline{k}$ defines an integral operator $\overline{K}$ on
$L^2(\overline{X},\overline{T},\overline{\mu})$ given by
$\overline{K}f(x)=\int_{\overline{X}}\overline{k}(x,y)f(y)\overline{\mu}(dy)$.
We shall denote by $\overline{k}^{(i)}$, the $i$-fold convolution of $\overline{k}$
with itself. Hence the integral operator associated with $\overline{k}^{(i)}$ is
$\overline{K}^{i}$. We shall write $\rho(\overline{k},\overline{\mu})$ to denote the
norm of $\overline{K}$ in $L^2(\overline{\mu})$.

A non-negative measurable function $\overline{\phi}$ on $\overline{X}$ is called
a weight function.

Given a type space $(\overline{X},\overline{T},\overline{\mu})$, a weight function
$\overline{\phi}$ and a sequence of kernels $\overline{k}_n$, we construct
a random graph by letting its vertices be points of a Poisson process $\poi$
on $\overline{X}$ with intensity $n\overline{\mu}(dx)$ and then join
any two points $x,y\in\poi$ with probability $\frac{\overline{k}_n(x,y)}{n}\wedge 1$.
The volume of a connected component $C$ is given by $\mathrm{volume}(C):=\sum_{x\in C}\overline{\phi}(x)$.

Then $IA_n(a,b,c)_v$ corresponds to the IRG model associated with $(X,\mu,k_{n,v},\phi_v)$
where $X=[0,T]\times W$ and
$W=D([0,T]:\mathbb{Z}_{\geq 0})$ is equipped with the Skorohod topology;
$\mu(d(s,w))=a(s)ds\ \nu_s(dw)$ where $\nu_s$ is the law of the function denoting
the size of a cluster born at time $s$ in $IA_n(a,b,c)$;
$k_{n,v}(x,y)=np_{n,v}(x,y)$ and $\phi_v(s,w)=w(v)$.

We also define the kernel $k_v$ by
$$k_v((s,w),(s_1,w_1)):=\int_0^v w_1(u)w_2(u)b(u)du.$$
It is clear that $k_{n,v}(x,y)\leq k_v(x,y)$ for $x,y\in X$.
The IRG model associated with $(X,\mu,k_v,\phi_v)$ will be denoted by $RG_{n,v}(a,b,c)$.
Note that for constructing $RG_{n,v}(a,b,c)$,
we have taken the constant sequence of kernels where each element is $k_v$.

Sometimes we shall write
$\mu(a,b,c)$ to express the dependence of $\mu$ on $a,b,c$
($\mu$ as a measure depends only on the functions $a$ and $c$, but we prefer to write it this way
to put emphasis on the underlying IRG model).
We shall write $\rho_v(a,b,c)$ to denote the norm of the operator $K_v$
associated with $k_v=k_v(a,b,c)$ in $L^2(\mu(a,b,c))$ (hence $\rho_v(a,b,c)=\rho(k_v(a,b,c),\mu(a,b,c))$).
\subsection{Proof of Theorem \ref{thm:continuous time BF}}
We present the proof of Theorem \ref{thm:continuous time BF} in this section.
Throughout the proof $C$, $C'$ etc.
will denote positive universal constants whose values may change from line to line.
Special constants will be indexed, as for example $C_1$, $C_2$ etc.
In the proof we shall assume $\lambda=1$ since it will not make a difference.

The first few steps consist of reducing the problem to getting
an upper bound on the total progeny of a continuum type branching process.
Assume that $t=t(n)$ satisfies $t\leq t_c-n^{-1/3}$.
Let $C_n^0(t)$ denote the component of the first doubleton appearing in $BF_n(t)$.
Fix $\gamma\in(1/3,1/2)$ and define $E_n:=\{\sup_{0\leq u\leq T} |x_n(u)-x(u)|>n^{-\gamma}\}$
where $T=2t_c$.
Since $a_0(x(u))$
is $C^1$ on $[0,T]$ and $|a_n(u)-a_0(u)|=O(1/n)$, it follows that on the event $E_n^c$
\begin{align}
a_n(x_n(u))&\leq a_0(x(u))+C_1(\frac{1}{n}+|x_n(u)-x(u)|)\nonumber\\
&\leq a_0(x(u))+\frac{C_2}{n^{\gamma}}\nonumber
\end{align}
and similar upper bounds hold for $b_n(x_n(u))$ and $c_n(x_n(u))$.
Set $\delta=\delta_n=C_2n^{-\gamma}$ and define $a_{n,\delta}(u)=(a_0(x(u))+\delta_n)$
for $t\in [0,T]$. Define $b_{n,\delta}(u)$ and $c_{n,\delta}(u)$ similarly.
Hence, $a_n(x_n(u))\leq a_{n,\delta}(u)$ on $E_n^c$ and similar upper bounds hold for
$b_n(x_n(u))$ and $c_n(x_n(u))$.
Note that
$$\sup _{[0,T]}\max\{a_{n,\delta}(u),b_{n,\delta}(u),c_{n,\delta}(u)\}\leq 1\text{ for large }n.$$

 Consider the first immigrating doubleton in $IA_n(a_{n,\delta},b_{n,\delta},c_{n,\delta})$
and let $C_{n,\delta}^{IA}(t)$ denote the size of the
component of $IA_n(a_{n,\delta},b_{n,\delta},c_{n,\delta})_t$ which contains the first immigrating doubleton.
Let $\nu_{s,\del}$ denote the measures on $W$ such that
$\mu(a_{n,\del},b_{n,\del},c_{n,\del})(d(s,w))=a_{n,\del}(s)ds\  \nu_{s,\del}(dw)$.
Consider the IRG model $RG_{n,t}(a_{n,\delta},b_{n,\delta},c_{n,\delta})$
conditioned on having a point $(0,w)$ where
$w$ is distributed according to $\nu_{0,\del}$. Let
$C_{n,\delta}^{RG}(t)$ be the volume of the component containing $(0,w)$.

The assertions in the next Lemma follow from Lemma 5.1, Lemma 5.2,
Lemma 6.1, Lemma 6.4 and Lemma 6.10 of \cite{bhamidi}.
\begin{lemma}\label{lem:bhamidi} The following hold.\\
(i) Bound on $\PR(E_n)$: $\PR(E_n)\leq \exp(-Cn^{1-2\gamma})$.\\
(ii) Connection between BF process, $IA_n(a_{n,\delta},b_{n,\delta},c_{n,\delta})$ and
$RG_{n,t}(a_{n,\delta},b_{n,\delta},c_{n,\delta})$: We have
$$\PR(L_1^{BF}(t)>m)\leq nT\PR(C_n^0(t)>m).$$
Further,
$$\PR(C_n^0(t)>m,E_n^c)\leq\PR(C_{n,\delta}^{IA}(t)>m)\leq \PR(C_{n,\delta}^{RG}(t)>m).$$\\
(iii) Properties of operator norms: The function $f(u):=\rho_{u}(a_0(x(\cdot)),b_0(x(\cdot)),c_0(x(\cdot)))$
is strictly increasing and satisfies $f(t_c)=1$. Further,
there exists some positive constant $\eta$ such that
$$(1-f(u))/(t_c-u)\to\eta\text{ as }u\uparrow t_c.$$
\end{lemma}
From Lemma \ref{lem:bhamidi}, we get
\begin{align}\label{eqn:1}
\PR(L_1^{BF}(t)>m)&\leq nT\left[\PR(C_n^0(t)>m,E_n^c)+\PR(E_n)\right]\\
& \leq nT\left[\PR(C_{n,\delta}^{RG}(t)>m)+\exp(-Cn^{1-2\gamma})\right]\nonumber\\
&=nT\left[\int_{W}\PR_{w_0}(C_{n,\delta}^{RG}(t)>m)\ \nu_{0,\del}(dw_0)+\exp(-Cn^{1-2\gamma})\right]\nonumber.
\end{align}
Here, $\PR_{w_0}(\cdot)=\PR(\cdot|\ (0,w_0)\in RG_{n,t}(a_{n,\delta},b_{n,\delta},c_{n,\delta}))$
for fixed $w_0\in W$.

Consider now a branching process on $[0,t]\times W$ as follows.
Write $k_{t,\delta}$ for $k_t(a_{n,\delta},b_{n,\delta},c_{n,\delta})$
and $\mu_{\del}$ for $\mu(a_{n,\delta},b_{n,\delta},c_{n,\delta})$.
Define $x_0=(0,w_0)$ to be generation
zero of the branching process. For $k\geq 0$, denote by $N_k$, the total
number of points in generation $k$. (Thus $N_0=1$.) For $k\geq 0$, each of the points
$x_1^{(k)},\hdots,x_{N_k}^{(k)}$ in generation $k$ gives birth to its own offsprings
according to a Poisson processes with intensities $k_{t,\delta}(x_i^{(k)},y)\ \mu_{\del}(dy)$ respectively
independent of the other points in generation $k$.

Let $G(x_0):=\sum_{k=0}^{\infty}\sum_{i=1}^{N_k}\phi_t(x_i^{(k)})$
denote the total progeny. By a breadth first search argument (see Lemma 6.12 in \cite{bhamidi}),
we have
\begin{equation}\label{eqn:4}
\PR_{w_0}(C_{n,\delta}^{RG}(t)>m)\leq \PR_{w_0}(G(x_0)>m).
\end{equation}
From \eqref{eqn:1} and \eqref{eqn:4}, we have
\begin{equation}\label{eqn:5}
\PR(L_1^{BF}(t)>m)\leq nT\left[\E_{\nu_{0,\del}}(\PR_{w_0}(G(x_0)>m))+\exp(-Cn^{1-2\gamma})\right].
\end{equation}

The following Lemma is an improvement over Lemma 6.9 of \cite{bhamidi}.
Recall that $T=2t_c$.
\begin{lemma}\label{lem:BF norm difference}
There exists a positive constant $\beta_0$ such that
$$|\rho_u(a_{n,\delta},b_{n,\delta},c_{n,\delta})-\rho_u(a,b,c)|\leq \beta_0\delta, \text{ for each }u\in[0,T]$$
where $a(\cdot)=a_0(x(\cdot))$, $b(\cdot)=b_0(x(\cdot))$ and $c(\cdot)=c_0(x(\cdot))$.
\end{lemma}
The proof of Lemma \ref{lem:BF norm difference} is given in the Section \ref{sec:Proof of Lemma for BF}.
An application of Lemma \ref{lem:BF norm difference}  together with $(iii)$ of Lemma \ref{lem:bhamidi}
yields
\begin{align}\label{eqn:6}
\rho_t(a_{n,\delta},b_{n,\delta},c_{n,\delta})&\leq \rho_t(a,b,c)+\frac{\beta_0 C_2}{n^{\gamma}}\\
&\leq 1-C(t_c-t)+\frac{\beta_0 C_2}{n^{\gamma}}\nonumber\\
&\leq 1-\beta(t_c-t)\nonumber
\end{align}
for a universal constant $\beta>0$. Here
we have used the fact that $t\leq t_c-n^{-1/3}$ and $\gamma>1/3$.
Let $\Delta=\Delta_{n,t}:=1-\rho_t(a_{n,\delta},b_{n,\delta},c_{n,\delta})$.
We have just shown that $\Delta\geq \beta(t_c-t)>0$.

Denote by $K_{t,\delta}$,
the integral operator associated with the kernel $k_{t,\delta}$. Let $X_t:=[0,t]\times W$.
Let $\overline{\mu}:=\mu(1,1,1)$, i.e. the measure on $[0,T]\times W$ associated with
the functions which are identically equal to one and let
$\overline{\nu}_s$ be the family of measures
on $W$ such that $\overline{\mu}(d(s,w))=ds\ \overline{\nu}_s(dw)$.

We shall use the following lemmas to get an upper bound on the right side of \eqref{eqn:5}.
\begin{lemma}\label{lem:2}
Define the function
$\phi_{\infty}:=\phi_t+\sum_{i=1}^{\infty}(1+\Delta/2)^i K_{t,\delta}^i\phi_t.$
Then for every $(s,w)\in [0,t]\times W$, we have
$$\phi_{\infty}(s,w)\leq \frac{C_3w(T)}{\Delta}$$
for some universal constant $C_3$.
\end{lemma}
Define the functions $f_{\ell}:=\phi_t+\sum_{i=1}^{\ell}(1+\Del/2)^iK_{t,\del}^i\phi_t$
for $\ell\geq 1$ and set $f_0=\phi_t$.
\begin{lemma}\label{lem:3}
Fix $x=(s,w)\in X_t$ and let $y_1,\hdots,y_N$ be the points of a Poisson process in $X_t$
having intensity $k_{t,\del}(x,y)\mu_{\del}(dy)$. Then there exists $\eta_0>0$ (independent of $x$) such that
for every $\eta\in(0,\eta_0]$, we have
$$\E\exp\left(\eta\Delta^2\sum_{i=1}^N f(y_i)\right)\leq \exp\left(\eta\Del^2(1+\Del/2)K_{t,\del}f(x)\right)$$
whenever $f\in\{f_{\ell}:\ell\geq 0\}$.
\end{lemma}
To simplify notations, throughout the rest of this section,
we shall write $K_{\del}$ (resp. $k_{\del}$) for $K_{t,\delta}$ (resp. $k_{t,\delta}$),
$\phi$ for $\phi_t$, $\rho$ for $(1-\Delta)$ and $a_{\del}$ (resp. $b_{\del}$, $c_{\del}$) for $a_{n,\del}$ (resp. $b_{n,\del},c_{n,\del}$).
\begin{proof}[Proof of Lemma \ref{lem:2}]
First note that
\begin{align}\label{eqn:lem2-1}
\|\phi\|_{L^2(\mu_{\del})}^2
&=\int_{X_t} w_1(t)^2 \mu_{\del}(d(s_1,w_1))\\
&\leq \int_0^t a_{\del}(s_1)\E_{\nu_{s_1,\del}}(w_1(T))^2 ds_1\nonumber\\
&\leq C\E_{\nu_{0,\del}}(w_1(T))^2\nonumber\\
&\leq C\E_{\overline{\nu}_{0}}(w_1(T))^2<\infty.\nonumber
\end{align}
Also
\begin{align}\label{eqn:lem2-2}
|K_{\del}^i\phi(s,w)|&=|\int_{X_t}k_{\del}^{(i)}((s,w),(s_1,w_1))\phi(s_1,w_1)\mu_{\del}(d(s_1,w_1))|\\
&\leq \|k_{\del}^{(i)}((s,w),.)\|_{L^2(\mu_{\del})}\|\phi\|_{L^2(\mu_{\del})}\nonumber\\
&\leq C\|K_{\del}^{i-1}k_{\del}((s,w),.)\|_{L^2(\mu_{\del})}\nonumber\\
&\leq C\rho^{i-1}\|k_{\del}((s,w),.)\|_{L^2(\mu_{\del})}.\nonumber
\end{align}
Now
\begin{align}\label{eqn:bound on k}
k_{\del}((s,w),(s_1,w_1))=\int_0^t w(u)w_1(u)b_{\del}(u)du\leq Cw(T)w_1(T).
\end{align}
Hence
\begin{align}\label{eqn:lem2-3}
\|k_{\del}((s,w),.)\|_{L^2(\mu_{\del})}^2&=\int_{X_t}k_{\del}((s,w),(s_1,w_1))^2\mu_{\del}(d(s_1,w_1))\\
&\leq C\int_0^t a_{\del}(s_1)ds_1\int_{w_1\in W}w^2(T)w_1^2(T)\nu_{s_1,\del}(dw_1)\nonumber\\
&\leq C'w^2(T)\E_{\overline{\nu}_0}(w_1(T)^2)\leq C''w^2(T).\nonumber
\end{align}
From \eqref{eqn:lem2-1}, \eqref{eqn:lem2-2} and \eqref{eqn:lem2-3}, we have
\begin{align}
\phi_{\infty}(s,w)&= \phi(s,w)+\sum_{i=1}^{\infty}(1+\Del/2)^i K_{\del}^i\phi(s,w)\nonumber\\
&\leq w(T)+C\sum_{i=1}^{\infty}(1+\Del/2)^i\rho^{i-1}w(T)\nonumber\\
&\leq C'w(T)\left(1+\frac{1}{1-(1-\Del)(1+\Del/2)}\right)\nonumber\\
&\leq C_3w(T)/\Del.\nonumber
\end{align}
\end{proof}
\begin{proof}[Proof of Lemma \ref{lem:3}]
We have
\begin{equation}\label{eqn:lem3-1}
\E\exp\left(\eta\Delta^2\sum_{i=1}^N f(y_i)\right)=\exp\left(\int_{X_t}k_{\del}(x,u)(\exp(\eta\Del^2f(u))-1)\mu_{\del}(du)\right).
\end{equation}
Now
\begin{align}\label{eqn:lem3-2}
&\int_{X_t}k_{\del}(x,u)(\exp(\eta\Del^2f(u))-1)\mu_{\del}(du)\\
&\hskip10pt=\int_{X_t}k_{\del}(x,(s_1,w_1))(\exp(\eta\Del^2f(s_1,w_1))-1)\mu_{\del}(d(s_1,w_1))\nonumber\\
&\hskip10pt=\int_{s_1=0}^t\int_{w_1\in W}\left(\int_{s_1}^t w(v)w_1(v)b_{\del}(v)dv\right)
(\exp(\eta\Del^2f(s_1,w_1))-1)\nu_{s_1,\del}(dw_1)a_{\del}(s_1)ds_1\nonumber\\
&\hskip10pt=\int_{s_1=0}^t\int_{v=s_1}^t w(v)b_{\del}(v)a_{\del}(s_1)\E_{\nu_{s_1,\del}}\left[w_1(v)(\exp(\eta\Del^2f(s_1,w_1))-1)\right]dv ds_1.\nonumber
\end{align}
Fix $s_1\in(0,t)$ and $v\in(s_1,t)$, then
\begin{align}\label{eqn:lem3-3}
&\E_{\nu_{s_1,\del}}\left[w_1(v)(\exp(\eta\Del^2f(s_1,w_1))-1)\right]\\
&\hskip10pt=\eta\Del^2\E_{\nu_{s_1,\del}}\left[w_1(v)\left(f+\eta\Del^2f^2/2!+
\sum_{j=2}^{\infty}(\eta\Del^2)^jf^{j+1}/(j+1)!\right)\right]\nonumber\\
&\hskip10pt=:\eta\Del^2\E_{\nu_{s_1,\del}}(T_1+T_2+T_3).\nonumber
\end{align}
Note that $\Del^{2j}\leq \Del^{j+2}$ for $j\geq 2$ and $f\leq\phi_{\infty}$. Hence from
Lemma \ref{lem:2}, we have
\begin{align}
\E_{\nu_{s_1,\del}}(T_3)&\leq \E_{\nu_{s_1,\del}}
\left[w_1(v)\sum_{j=2}^{\infty}(\eta\Del^2)^j\frac{(C_3w_1(T))^{j+1}}{\Del^{j+1}(j+1)!}\right]\nonumber\\
&\leq \E_{\nu_{s_1,\del}}
\left[w_1(T)\sum_{j=2}^{\infty}\eta^j\Del\frac{(C_3w_1(T))^{j+1}}{(j+1)!}\right]\nonumber\\
&\leq C_3^2\eta\Del\E_{\nu_{s_1,\del}}
\left[w_1(T)^3\sum_{j=2}^{\infty}\frac{(C_3\eta w_1(T))^{j-1}}{(j+1)!}\right]\nonumber\\
&\leq C_3^2\eta\Del\E_{\overline{\nu}_0}\left(w_1(T)^3\exp(C_3\eta w_1(T))\right)\nonumber\\
&\leq C_3^2\eta\Del\left(\E_{\overline{\nu}_0}(w_1(T)^6)\right)^{1/2}
\left(\E_{\overline{\nu}_0}(\exp(2C_3\eta w_1(T)))\right)^{1/2}.\nonumber
\end{align}
Since $w_1$ has an exponentially decaying tail (Lemma 6.7 in \cite{bhamidi}),
we can choose $\eta_1>0$ small so that $\E_{\overline{\nu}_0}(\exp(2C_3\eta_1 w_1(T))<\infty$
and for $0\leq\eta\leq\eta_1$
\begin{equation}\label{eqn:lem3-4}
\E_{\nu_{s_1,\del}}(T_3)\leq \Del\leq\frac{\Del}{4}\E_{\nu_{s_1,\del}}(w_1(v)f(s_1,w_1)),
\end{equation}
the last inequality follows by noting that for $v\in(s_1,t)$, $w_1(v)\geq 2$
and $f(s_1,w_1)\geq \phi(s_1,w_1)=w_1(t)\geq 2$ when $w_1$ is distributed as $\nu_{s_1,\del}$.

Using the bound $f(s_1,w_1)\leq C_3w_1(T)/\Del$, we get
\begin{align}\label{eqn:lem3-5}
\E_{\nu_{s_1,\del}}T_2&= \E_{\nu_{s_1,\del}}(w_1(v)\eta\Del^2f^2/2)\\
&\leq \frac{C_3\eta\Del}{2} \E_{\nu_{s_1,\del}}(w_1(T)^2f(s_1,w_1)).\nonumber
\end{align}
Suppose $f=f_{\ell}=\phi+\sum_{i=1}^{\ell}(1+\Del/2)^iK_{\del}^i\phi$.
Choose an $\eta_2>0$ so that $C_3\eta_2\E_{\overline{\nu}_0}(w_1(T)^3)\leq 2$, then for $0<\eta\leq\eta_2$,
\begin{align}\label{eqn:lem3-6}
\frac{C_3\eta\Del}{2}\E_{\nu_{s_1,\del}}(w_1(T)^2\phi(s_1,w_1))&\leq \frac{C_3\eta\Del}{2}\E_{\overline{\nu}_{0}}(w_1(T)^3)\\
&\leq \Del\leq \frac{\Del}{4}\E_{\nu_{s_1,\del}}(w_1(v)\phi(s_1,w_1)).\nonumber
\end{align}
For $i\geq 1$, let $g=K_{\del}^{i-1}\phi$. Then for $0<\eta\leq\eta_2$,
\begin{align}\label{eqn:lem3-7}
&\frac{C_3\eta\Del}{2}\E_{\nu_{s_1,\del}}(w_1(T)^2K_{\del}^i\phi(s_1,w_1))\\
&\hskip10pt=\frac{C_3\eta\Del}{2}\E_{\nu_{s_1,\del}}(w_1(T)^2K_{\del}g(s_1,w_1))\nonumber\\
&\hskip10pt=\frac{C_3\eta\Del}{2}
\E_{\nu_{s_1,\del}}\left[w_1(T)^2\int_{X_t}\left(\int_{z=s_1}^tw_1(z)w_2(z)b_{\del}(z)dz\right)g(s_2,w_2)\mu_{\del}(d(s_2,w_2))
\right]\nonumber\\
&\hskip10pt=\frac{C_3\eta\Del}{2} \int_{X_t}\left[\left(\int_{z=s_1}^t\E_{\nu_{s_1,\del}}
\left(w_1(T)^2w_1(z)\right)w_2(z)b_{\del}(z)dz\right)g(s_2,w_2)\right]\mu_{\del}(d(s_2,w_2))
\nonumber\\
&\hskip10pt\leq\frac{C_3\eta\Del}{2} \int_{X_t}
\left[\left(\int_{z=s_1}^t\E_{\overline{\nu_0}}\left(w_1(T)^3\right)w_2(z)b_{\del}(z)dz\right)g(s_2,w_2)\right]\mu_{\del}(d(s_2,w_2))
\nonumber\\
&\hskip10pt\leq\Del \int_{X_t}
\left[\left(\int_{z=s_1}^t w_2(z)b_{\del}(z)dz\right)g(s_2,w_2)\right]\mu_{\del}(d(s_2,w_2))
\nonumber\\
&\hskip10pt\leq\frac{\Del}{4} \int_{X_t}
\left[\left(\int_{z=s_1}^t\E_{\nu_{s_1,\del}}\left(w_1(v)w_1(z)\right)w_2(z)b_{\del}(z)dz\right)g(s_2,w_2)\right]\mu_{\del}(d(s_2,w_2))
\nonumber\\
&\hskip10pt=\frac{\Del}{4}
\E_{\nu_{s_1,\del}}\left[w_1(v)\int_{X_t}k_{\del}((s_1,w_1),(s_2,w_2))g(s_2,w_2)\mu_{\del}(d(s_2,w_2))\right]
\nonumber\\
&\hskip10pt=\frac{\Del}{4} \E_{\nu_{s_1,\del}}\left(w_1(v)K_{\del}^i\phi(s_1,w_1)\right).
\nonumber
\end{align}
From \eqref{eqn:lem3-5}, \eqref{eqn:lem3-6} and \eqref{eqn:lem3-7}, we get
\begin{equation}\label{eqn:lem3-8}
\E_{\nu_{s_1,\del}}T_2\leq\frac{\Del}{4} \E_{\nu_{s_1,\del}}\left(w_1(v)f(s_1,w_1)\right).
\end{equation}
From \eqref{eqn:lem3-3}, \eqref{eqn:lem3-4} and \eqref{eqn:lem3-8}, we have
\begin{align}\label{eqn:lem3-9}
\E_{\nu_{s_1,\del}}\left[w_1(v)(\exp(\eta\Del^2f(s_1,w_1))-1)\right]
\leq \eta\Del^2(1+\frac{\Del}{2})\E_{\nu_{s_1,\del}}(w_1(v)f(s_1,w_1)).
\end{align}
From \eqref{eqn:lem3-9} and \eqref{eqn:lem3-2}, we get
\begin{align}
&\int_{X_t}k_{\del}(x,u)(\exp(\eta\Del^2f(u))-1)\mu_{\del}(du)\\
&\hskip10pt\leq\eta\Del^2(1+\frac{\Del}{2})
\int_{s_1=0}^t\int_{v=s_1}^t w(v)b_{\del}(v)a_{\del}(s_1)\E_{\nu_{s_1,\del}}\left[w_1(v)f(s_1,w_1)\right]dv ds_1\nonumber\\
&\hskip10pt=\eta\Del^2(1+\frac{\Del}{2})
\int_{s_1=0}^t a_{\del}(s_1)\E_{\nu_{s_1,\del}}\left[k_{\del}(x,(s_1,w_1))f(s_1,w_1)\right] ds_1\nonumber\\
&\hskip10pt= \eta\Del^2(1+\Del/2)K_{\del}f(x),\nonumber
\end{align}
for $0<\eta\leq \eta_0:=\eta_1\wedge\eta_2$. This together with \eqref{eqn:lem3-1} yields the result.
\end{proof}
Now the proof of Theorem \ref{thm:continuous time BF} becomes routine and
can be finished in the same way as in Lemma 6.15 in \cite{bhamidi}.
We include the short argument.

Denote by $G_i:=\sum_{j=1}^{N_i}\phi(x_j^{(i)})$, the total volume of points
in generation $k$ and let $\mathcal{F}_k$ be the $\sigma$-field generated
by $\{x_i^{(j)}:\ i\leq N_j, j\leq k\}$. We make the following
\begin{claim*} For $0<\eta\leq\eta_0$,
$$\E\left(\exp\left(\eta\Del^2\sum_{i=j}^k G_i\right)|\mathcal{F}_j\right)
\leq \exp\left(\eta\Del^2\sum_{i=1}^{N_j}f_{k-j}(x_i^{(j)})\right).$$
\end{claim*}
\begin{proof}[Proof of claim]
The assertion is immediate for $j=k$. Assume it is true for $\ell+1\leq j\leq k$. Then
\begin{align}
&\E\left(\exp\left(\eta\Del^2\sum_{i=\ell}^k G_i\right)|\mathcal{F}_{\ell}\right)\nonumber\\
&=\exp(\eta\Del^2G_{\ell})\E\left[\E\left(\exp\left(\eta\Del^2\sum_{i=\ell+1}^k G_i\right)|\mathcal{F}_{\ell+1}\right)|\mathcal{F}_{\ell}\right]\nonumber\\
&\leq \exp(\eta\Del^2G_{\ell})
\E\left[\exp\left(\eta\Del^2\sum_{i=1}^{N_{\ell+1}}f_{k-\ell-1}(x_i^{(\ell+1)})\right)|\mathcal{F}_{\ell}\right]
\text{ (induction hypothesis)}\nonumber\\
&\leq\exp(\eta\Del^2G_{\ell})\exp\left(\eta\Del^2(1+\Del/2)\sum_{i=1}^{N_{\ell}}K_{\del}f_{k-\ell-1}(x_i^{(\ell)})\right)
\text{ (Lemma \ref{lem:3})}\nonumber\\
&=\exp\left(\eta\Del^2\sum_{i=1}^{N_{\ell}}f_{k-\ell}(x_i^{(\ell)})\right).\nonumber
\end{align}
\end{proof}
Setting $j=0$ and letting $k$ tend to infinity in this inequality, we get
\begin{align}
\E_{w_0}(\exp(\eta\Del^2G(x_0)))&\leq \exp(\eta\Del^2\phi_{\infty}((0,w_0)))\nonumber\\
&\leq\exp(\eta\Del C_3w_0(T))\leq\exp(C_3\eta w_0(T)).\nonumber
\end{align}
Hence we have
\begin{align}\label{eqn:7}
\E_{\nu_{0,\del}}\PR_{w_0}(G(x_0)>m))
&\leq \exp(-\eta_0\Del^2m)\E_{\nu_{0,\del}}\left[\E_{w_0}(\exp(\eta\Del^2G(x_0)))\right]\\
&\leq \exp(-\eta_0\Del^2m)\E_{\overline{\nu}_0}(\exp(C_3\eta w_0(T))\nonumber\\
&=C_4\exp(-\eta_0\Del^2m)\leq C_4\exp(-\eta_0\beta^2(t_c-t)^2m) ,\nonumber
\end{align}
since $\Delta\geq\beta(t_c-t)$. The constant $C_4$ is finite by the choice of $\eta_0$.

From \eqref{eqn:7} and \eqref{eqn:5}, we see that
$$\PR\left(L_1^{BF}(t)>\frac{2\log n}{\eta_0\beta^2(t_c-t)^2}\right)\stackrel{n\to\infty}{\longrightarrow} 0,$$
this completes the proof of Theorem \ref{thm:continuous time BF}.
\subsection{Proof of Lemma \ref{lem:BF norm difference}}\label{sec:Proof of Lemma for BF}
Throughout this section we shall write $a(u)$ (resp. $a_{\del}(u)$), $b(u)$ (resp. $b_{\del}(u)$)
and $c(u)$ (resp. $c_{\del}(u)$)
to mean $a_0(x(u))$ (resp. $a_{n,\del}(x(u))$), $b_0(x(u))$ (resp. $b_{n,\del}(x(u))$)
and $c_0(x(u))$ (resp. $c_{n,\del}(x(u))$). We write $\mu$, $k_u$ for $\mu(a,b,c)$ and
$k_u(a,b,c)$
and recall that $\mu_{\del}$ and $\overline{\mu}$ stand for $\mu(a_{\del},b_{\del},c_{\del})$
and $\mu(1,1,1)$ respectively.

Then
\begin{align}\label{eqn:lem1-1}
&|\rho_u(a,b,c)-\rho_u(a_{\del},b_{\del},c_{\del})|\\
&\hskip10pt
\leq |\rho_u(a,b,c)-\rho_u(a_{\del},b,c_{\del})|
+|\rho_u(a_{\del},b,c_{\del})-\rho_u(a_{\del},b_{\del},c_{\del})|\nonumber\\
&\hskip10pt=:T_1+T_2.\nonumber
\end{align}
From Lemma 6.19 of \cite{bhamidi}, it follows that
$\mu<<\mu_{\del}$ and for $x=(s,w)\in X$
$$f(x):=\frac{d\mu}{d{\mu}_{\del}}(x)=\frac{a(s)}{a_{\del}(s)}\exp\left(\int_{s}^Tw(u)(c_{\del}(u)-c(u))du\right)
\prod_{i=1}^{w(T)-2}\left(\frac{c(\gamma_i)}{c_{\del}(\gamma_i)}\right),$$
where $\gamma_i=\gamma_i(s,w)$ is the time of the $i$th jump of $w$ after time $s$.

Since $\rho(k_u(x,y),\mu)=\rho(k_u(x,y)\sqrt{f(x)f(y)},\mu_{\del})$, we have
\begin{align}\label{eqn:lem1-2}
T_1&=|\rho\left(k_u(x,y)\sqrt{f(x)f(y)},\mu_{\del}\right)-\rho(k_u(x,y),\mu_{\del})|\\
&\leq \rho\left(k_u(x,y)\left(\sqrt{f(x)f(y)}-1\right),\mu_{\del}\right)\nonumber\\
&\leq \sqrt{\int_X\int_X k_u^2(x,y)\left(\sqrt{f(x)f(y)}-1\right)^2\mu_{\del}(dx)\mu_{\del}(dy)}.\nonumber
\end{align}
Writing $x=(s_1,w_1)$, $y=(s_2,w_2)$ and $\tau_i=\gamma_i(x)$, $\xi_j=\gamma_j(y)$ we have
$L\leq \sqrt{f(x)f(y)}\leq U$, where
$$L:=\left(\frac{a(s_1)}{a_{\del}(s_1)}\frac{a(s_2)}{a_{\del}(s_2)}
\prod_{i=1}^{w_1(T)-2}\frac{c(\tau_i)}{c_{\del}(\tau_i)}
\prod_{j=1}^{w_2(T)-2}\frac{c(\xi_j)}{c_{\del}(\xi_j)}\right)^{1/2}$$
and
$$U:=\exp\left(\frac{\del T}{2}(w_1(T)+w_2(T))\right).$$
From \eqref{eqn:bound on k} and \eqref{eqn:lem1-2},
\begin{align}\label{eqn:lem1-3}
T_1&\leq C\sqrt{\int_X\int_X w_1(T)^2w_2(T)^2\left[(U-1)^2+(L-1)^2\right]\mu_{\del}(dx)\mu_{\del}(dy)}\\
&\leq C\sqrt{\int_X\int_X w_1(T)^2w_2(T)^2(U-1)^2\mu_{\del}(dx)\mu_{\del}(dy)}\nonumber\\
&\hskip10pt+ C\sqrt{\int_X\int_X w_1(T)^2w_2(T)^2(L-1)^2\mu_{\del}(dx)\mu_{\del}(dy)}\nonumber\\
&\hskip10pt=:C(T_3+T_4).\nonumber
\end{align}
Note that
\begin{align}
(U-1)^2&\leq 2\left(\exp\left(\frac{\del Tw_1(T)}{2}\right)-1\right)^2\exp\left(\del Tw_2(T)\right)\nonumber\\
&\hskip20pt +2\left(\exp\left(\frac{\del Tw_2(T)}{2}\right)-1\right)^2\exp\left(\del Tw_1(T)\right).\nonumber
\end{align}
This together with the simple inequality $e^x-1\leq xe^x$ yields
\begin{align}\label{eqn:lem1-4}
T_3^2&\leq 4\int\int w_1^2(T)w_2^2(T)
\left(\exp\left(\frac{\del Tw_1(T)}{2}\right)-1\right)^2\exp\left(\del Tw_2(T)\right)\mu_{\del}(dx)\mu_{\del}(dy)\\
&\leq C\del^2\left[\int w_1^4(T)\exp\left(\del Tw_1(T)\right)\mu_{\del}(dx)\right]
\left[\int w_2^2(T)\exp\left(\del Tw_2(T)\right)\mu_{\del}(dy)\right]\nonumber\\
&\leq C'\del^2\E_{\overline{\nu}_0}\left[w_1^4(T)\exp\left(\del Tw_1(T)\right)\right]
\E_{\overline{\nu}_0}\left[w_2^2(T)\exp\left(\del Tw_2(T)\right)\right].\nonumber
\end{align}
Since the tail of $w_1(T)$ decays exponentially, we conclude that for large $n$
(so that $\del=\del_n$ is sufficiently small)
\begin{equation}\label{eqn:T3 term}
T_3\leq C\del.
\end{equation}
Next
\begin{align}\label{eqn:T4 term decomposition}
T_4^2\leq C\int\int w_1^2(T)w_2^2(T)(L_1+L_2+L_3+L_4)\mu_{\del}(dx)\mu_{\del}(dy),
\end{align}
where
\begin{align}
L_1=\left(1-\sqrt{\frac{a(s_1)}{a_{\del}(s_1)}}\right)^2, &
L_2=\left(1-\sqrt{\frac{a(s_2)}{a_{\del}(s_2)}}\right)^2, \nonumber\\
L_3=\left(1-\sqrt{\prod_{i=1}^{w_1(T)-2}\frac{c(\tau_i)}{c_{\del}(\tau_i)}}\right)^2, &
L_4=\left(1-\sqrt{\prod_{j=1}^{w_2(T)-2}\frac{c(\xi_j)}{c_{\del}(\xi_j)}}\right)^2 \nonumber.
\end{align}
Note that $\inf_{[0,T]}a(s)=m_1>0$ and $0\leq a_{\del}(s)-a(s)\leq \del$.
Hence,
\begin{align}\label{eqn:L1 term bound}
&\int\int w_1^2(T)w_2^2(T)L_1\mu_{\del}(dx)\mu_{\del}(dy)\\
&\hskip15pt\leq \int\int w_1^2(T)w_2^2(T)\left(1-\frac{a(s_1)}{a_{\del}(s_1)}\right)^2\mu_{\del}(dx)\mu_{\del}(dy)\nonumber\\
&\hskip15pt\leq \frac{\del^2}{m_1^2}\int\int w_1^2(T)w_2^2(T)\mu_{\del}(dx)\mu_{\del}(dy)
\leq C\del^2\left(\E_{\overline{\nu}_0}w_1^2(T)\right)^2.\nonumber
\end{align}
The integrand corresponding to $L_2$ can be handled in the same way.

Now $\inf_{[0,\eps]}(c(t))'=\inf_{[0,\eps]}c_0'(x(t))x'(t)>0$ (see \cite{spencer}) for
some small $\eps>0$
and hence
\begin{equation}\label{eqn:cdel lower bound1}
c_{\del}(t)\geq \max(\del,m_2t)\text{ for }t\in[0,T]
\end{equation}
for a positive constant $m_2$.
 Further, $c(t)$ is increasing in an
interval $[0,t_0]$ and is bounded away from zero on $[t_0,T]$. Hence
\begin{equation}\label{eqn:cdel lower bound2}
c_{\del}(\tau_i)\geq \min(c_{\del}(\tau_1),m_3)\text{ for }1\leq i\leq w_1(T)-2.
\end{equation}
Hence, on the set $w_1(T)\geq 3$,
\begin{align}
L_3&\leq w_1(T)\sum_{i=1}^{w_1(T)-2}\left(1-\sqrt{\frac{c(\tau_i)}{c_{\del}(\tau_i)}}\right)^2\nonumber\\
&\leq w_1(T)\sum_{i=1}^{w_1(T)-2}\left(1-\frac{c(\tau_i)}{c_{\del}(\tau_i)}\right)^2\nonumber\\
&\leq w_1(T)\sum_{i=1}^{w_1(T)-2}\frac{\del^2}{c_{\del}^2(\tau_i)}\nonumber\\
&\leq w_1(T)\del^{2}\sum_{i=1}^{w_1(T)-2}\left(\frac{1}{m_3}+\frac{1}{c_{\del}(\tau_1)}\right)^{2}
\text{ (from \eqref{eqn:cdel lower bound2})}\nonumber\\
&\leq C\del^{2}w_1(T)^2\frac{1}{c_{\del}(\tau_1)^{2}}.\nonumber
\end{align}
Hence,
\begin{align}\label{eqn:L3bound 1}
&\int\int w_1^2(T)w_2^2(T)L_3\mu_{\del}(dx)\mu_{\del}(dy)\\
&\hskip10pt\leq C\del^{2}
\int\int w_2^2(T)\frac{w_1^4(T)}{c_{\del}(\tau_1)^{2}}\I_{\{w_1(T)\geq3\}}\mu_{\del}(dx)\mu_{\del}(dy)\nonumber\\
&\hskip10pt\leq C'\del^{2}\int \frac{w_1^4(T)}{c_{\del}(\tau_1)^{2}}\I_{\{w_1(T)\geq3\}}\mu_{\del}(dx)\nonumber\\
&\hskip10pt\leq C'\del^{2}\left(\int w_1^{4p}(T)\mu_{\del}(dx)\right)^{1/p}
\left(\int \frac{1}{c_{\del}(\tau_1)^{2q}}\I_{\{w_1(T)\geq3\}}\mu_{\del}(dx)\right)^{1/q}\nonumber\\
&\hskip10pt\leq C''\del^{2}
\left(\int \frac{1}{c_{\del}(\tau_1)^{2q}}\I_{\{w_1(T)\geq3\}}\mu_{\del}(dx)\right)^{1/q}.\nonumber
\end{align}
Here we choose $p,q>1$ so that $p^{-1}+q^{-1}=1$ and $2q=2+\theta$ with $0<\theta<1$.
Define
\begin{align}
\overline{\tau}_1=
\left\{
\begin{array}{l}
\tau_1, \text{ if }w_1(T)\geq 3,\\
T,\text{ if }w_1(T)=2.
\end{array}
\right.\nonumber
\end{align}
Then
\begin{align}\label{eqn:L3bound 2}
&\int \frac{1}{c_{\del}(\tau_1)^{2q}}\I_{\{w_1(T)\geq3\}}\mu_{\del}(dx)\\
&\hskip10pt\leq \int \frac{1}{c_{\del}(\overline{\tau}_1)^{2+\theta}}\mu_{\del}(dx)\nonumber\\
&\hskip10pt=\int_{s_1=0}^T a_{\del}(s_1)ds_1\E_{\nu_{s_1,\del}}\left[\frac{1}{c_{\del}(\overline{\tau}_1)^{2+\theta}}\right]\nonumber\\
&\hskip10pt\leq\int_{s_1=0}^T ds_1
\left[\int_{u=s_1}^T\frac{1}{c_{\del}(u)^{2+\theta}}
\exp\left(-2\int_{s_1}^u c_{\del}(z)dz\right)2c_{\del}(u)du+\frac{\nu_{s_1,\del}\{w_1(T)=2\}}{c_{\del}(T)^{2+\theta}}\right]\nonumber\\
&\hskip20pt\leq 2\int_{s_1=0}^T ds_1
\left[\int_{u=s_1}^T\frac{du}{(m_2u)^{1+\theta}}+\frac{1}{(m_2T)^{2+\theta}}\right].\nonumber
\end{align}
The last inequality is a consequence of \eqref{eqn:cdel lower bound1}. Since $\theta<1$, the
last integral is finite. A similar analysis can be carried out for the integrand
corresponding to $L_4$. Combining \eqref{eqn:T3 term}, \eqref{eqn:L1 term bound}, \eqref{eqn:L3bound 1}
and \eqref{eqn:L3bound 2}, we get
\begin{equation}\label{eqn:T1 bound}
T_1\leq C\del.
\end{equation}
Finally,
\begin{align}\label{eqn:T2 bound}
T_2^2&\leq \int_X\int_X\left(k_{u,\del}(x,y)-k_u(x,y)\right)^2\mu_{\del}(dx)\mu_{\del}(dy)\\
&\leq \int_X\int_X\left(\int_0^T w_1(z)w_2(z)(b_{\del}(z)-b(z)) dz\right)^2\mu_{\del}(dx)\mu_{\del}(dy)\nonumber\\
&\leq C\del^2\left(\int w_1^2(T)\overline{\mu}(dx)\right)^2=C'\del^2.\nonumber
\end{align}
From \eqref{eqn:T1 bound} and \eqref{eqn:T2 bound}, we get
$$|\rho_u(a,b,c)-\rho_u(a_{\del},b_{\del},c_{\del})|\leq C\del$$
which is the desired bound.
\subsection{Proof of Corollary \ref{cor:discrete time BF}}
Let $X_n(s)$ denote the number of edges in $BF_n(s)$. Define
a process $\overline{BF}_n(\cdot)$ by $\overline{BF}_n(2X_n(s)/n):=BF_n(s)$ for $s\geq 0$
and extend the definition to $\R_+$ by right continuity. Then $\overline{BF}_n(\cdot)$
has the same distribution as $DBF_n(\cdot)$. Let $L_1^{\overline{BF}}(s)$ denote
the size of the largest component of $\overline{BF}_n(s)$.
Let us assume that $t_c/2\leq t\leq t_c-\lambda n^{-1/3}$, since for $t\leq t_c/2$
the desired bound will follow directly. We have
\begin{align*}
\PR\left(L_1^{BF}\left(t+\frac{\log n}{\sqrt{n}}\right)\geq m\right)
&\geq \PR\left(L_1^{BF}\left(t+\frac{\log n}{\sqrt{n}}\right)\geq m,\
X_n\left(t+\frac{\log n}{\sqrt{n}}\right)\geq \frac{nt}{2}\right)\\
&\geq \PR\left(L_1^{\overline{BF}}(t)\geq m,\
 X_n\left(t+\frac{\log n}{\sqrt{n}}\right)\geq \frac{nt}{2}\right).
\end{align*}
Hence
\begin{align}\label{eqn:DBF decomposition}
\PR(L_1^{\overline{BF}}(t)\geq m)\leq
\PR(L_1^{BF}(t+\log n/\sqrt{n})\geq m)+
\PR( X_n(t+\log n/\sqrt{n})< nt/2).
\end{align}

Let $Z_1,\hdots,Z_n$ be i.i.d. Poisson random variables with mean $\mu_n:=\frac{1}{2}(1-1/n)^2(t+\log n/\sqrt{n})$.
Then $X_n(t+\log n/\sqrt{n})\stackrel{d}{=}Z_1+\hdots+Z_n$. Since $nt/2-n\mu_n=-\sqrt{n}\log n/2+O(1)$, we have
\begin{align}\label{eqn:BE}
\PR(X_n(t+\log n/\sqrt{n})< nt/2)&\leq \PR(\sum_{j=1}^n(Z_j-\mu_n)/\sqrt{n\mu_n}<-C\log n)\\
&\leq \Phi(-C\log n)+C'/\sqrt{n},\nonumber
\end{align}
the last inequality being a consequence of the classical Berry-Ess\'{e}en theorem (\cite{berry-esseen}).
The result now follows from \eqref{eqn:DBF decomposition}, \eqref{eqn:BE} and Theorem \ref{thm:continuous time BF}.
\section{Proofs for general bounded size rules}\label{sec:proofs1}
As in the case of the Bohman-Frieze process,
the proofs for general bounded size rules also proceed through the analysis of an inhomogeneous
random graph model. We start with some definitions.

Fix $K\geq 2$ and $F\subset\Omega_K^4$. Consider the bounded size rule associated with $F$.
For $\vec{j}=(j_1,j_2,j_3,j_4)\in F$ and $i\in\Omega_K$, let
\begin{align*}
\Delta(\vec{j},i)=
\left\{
\begin{array}{l}
\frac{i}{2}\cdot(\I\{j_1+j_2=i\}-\I\{j_1=i\}-\I\{j_2=i\}), \text{ if }1\leq i\leq K,\\
\frac{1}{2}\cdot\I\{j_1+j_2>K\}\cdot(j_1\I\{j_1\leq K\}+j_2\I\{j_2\leq K\}),\text{ if }i=\omega.
\end{array}
\right.
\end{align*}
When $\vec{j}\in F^c$, we replace $j_1,j_2$ in the above expressions by $j_3,j_4$ respectively
and define $\Delta(\vec{j},i)$ in the same way. For the motivation behind this definition,
see \cite{spencer}. Consider now the system of ODEs
\begin{equation}\label{eqn:ODE}
\overline{x}_i'(v)=\sum_{\vec{j}\in\Omega_K^4}\Delta(\vec{j},i)
\overline{x}_{j_1}(v)\overline{x}_{j_2}(v)\overline{x}_{j_3}(v)\overline{x}_{j_4}(v),\ i\in\Omega_K
\end{equation}
with the initial condition $\overline{x}_{i}(0)=\I\{i=1\}$.
Spencer and Wormald \cite{spencer} showed that this system has a solution
$\{x_i(\cdot):\ i\in\Omega_K\}$ defined for $v\geq 0$. Further, the proportion
of vertices of $DBSR_n(v)$ in components of size $i$ converges to $x_i(v)$
in probability. For $i_1,i_2\in\Omega_K$, we define the functions
$F_{i_1,i_2}:[0,1]^{K+1}\to\R$ by
\begin{equation}\label{eqn:definition F}
F_{i_1,i_2}(x_1,\hdots,x_K,x_{\omega})=\frac{1}{2}
\big(\sum_{\substack{\vec{j}\in F:\\ \{j_1,j_2\}=\{i_1,i_2\}}}
x_{j_1}x_{j_2}x_{j_3}x_{j_4}+\sum_{\substack{\vec{j}\in F^c:\\ \{j_3,j_4\}=\{i_1,i_2\}}}
x_{j_1}x_{j_2}x_{j_3}x_{j_4}\big).
\end{equation}

We now define some rate functions which are analogous to the functions $a_0(x(\cdot)), b_0(x(\cdot))$
and $c_0(x(\cdot))$ appearing in the proof of Theorem \ref{thm:continuous time BF}.
For $i\in \{1,\hdots,K\}$, let
\begin{align}\label{eqn:definition a}
a_i(v)=\sum_{\substack{1\leq i_1\leq i_2\leq K,\\i_1+i_2=K+i}}F_{i_1,i_2}(x_1(v),\hdots,x_K(v),x_{\omega}(v)),\
c_i(v)=\frac{1}{x_{\omega}(v)} F_{i,\omega}(x_1(v),\hdots,x_K(v),x_{\omega}(v)),
\end{align}
and let
\begin{align}\label{eqn:definition b}
b(v)= F_{\omega,\omega}(x_1(v),\hdots,x_K(v),x_{\omega}(v))/x_{\omega}^2(v).
\end{align}

Let $BSR_n^{\ast}(v)$ be the subgraph of $BSR_n(v)$ consisting of
those components of $BSR_n(v)$ which contain at least $K+1$ vertices.
Then, for $1\leq i\leq K$, $na_i(v)$ approximates the rate at which a component of size $K+i$
appears in $BSR_n^{\ast}(v)$, i.e. the rate at which two components of sizes smaller than
$K+1$ merge to form a component of size $K+i$.
Given two vertices in $BSR_n^{\ast}(v)$, $b(v)/n$ approximates the rate
at which these two vertices are joined by an edge. Given a
vertex in $BSR_n^{\ast}(v)$ and $1\leq i\leq K$,
$c_i(v)$ approximates the rate at which a component of size $i$
gets linked to this vertex by an edge.
For the actual computations leading to these formulas for the
rate functions, see \cite{bhamidi2}.

Next, fix $\gamma\in(0,1/2)$ and let $\del=\del_n:=n^{-\gamma}$. Define,
$a_{i,\del}(v)=a_i(v)+\del_n$ and $c_{i,\del}(v)=c_i(v)+\del_n$ for $1\leq i\leq K$.
Similarly, define $b_{\del}(v)=b(v)+\del_n$.
Since $x_i(t)\geq 0$ and $\sum_{i\in\Omega_K}x_i(t)=1$ for $t\geq 0$ (see \cite{spencer}),
it follows that $\sup\{a_{i,\del}(v):\ 1\leq i\leq K, v\in[0,T]\}\leq 1$
for large $n$ and similar inequalities are true for $\{c_{i,\del}\}_{i\leq K}$ and $b_{\del}$.
Let $t_c$ be the critical time
for the bounded size rule. Let $T=2t_c$ and $X=[0,T]\times W$
where $W=D([0,T]:\Z_{\geq 0})$ is equipped with the Skorohod topology.
For $1\leq i\leq K$ and $s\in[0,T]$, let $\nu_{i,s;\del}$ be the law of the Markov process
$Z(\cdot)$ in $W$ which satisfies $Z(u)=0$ for $0\leq u<s$ and $Z(s)=K+i$ and
conditional on $\{Z(u)\}_{u\leq s_1}$, the process makes a positive jump of
size $j$ in $(s_1,s_1+ds_1]$ with rate $Z(s_1)c_{j,\del}(s_1)$ for $1\leq j\leq K$.
The jumps of $Z(\cdot)$ happen independently across $j$.
Let $\mu_{\del}$ be the measure on $X$ given by
$$\mu_{\del}(d(s,w))=\sum_{i=1}^K a_{i,\del}(s)ds\ \nu_{i,s;\del}(dw).$$
For $v\in[0,T]$, define the kernel $k_{v,\del}$ on $X\times X$ by
$k_{v,\del}((s_1,w_1),(s_2,w_2)):=\int_0^v w_1(u)w_2(u)b_{\del}(u)du$.
Let $\phi_v:X\to\R$ be the weight function given by $\phi_v(s,w)=w(v)$.
As in Section \ref{sec:connection}, we construct an IRG model
$RG_{n,v}=RG_{n,v}(\{a_{i,\del}\}_{i\leq K},b_{\del},\{c_{i,\del}\}_{i\leq K})$
associated with $(X,\mu_{\del},k_{v,\del},\phi_v)$.

In general, for any nonnegative continuous functions $\{\overline{a}_i\}_{1\leq i\leq K}$,
$\overline{b}$ and $\{\overline{c}_i\}_{1\leq i\leq K}$ and $v\in[0,T]$, we denote by
$k_v(\{\overline{a}_i\}_{i\leq K},\overline{b},\{\overline{c}_i\}_{i\leq K})$
(resp. $\mu(\{\overline{a}_i\}_{i\leq K},\overline{b},\{\overline{c}_i\}_{i\leq K})$),
the kernel (resp. the measure) on $X$ constructed in a manner similar to the construction
of $k_{v,\del}$ (resp. $\mu_{\del}$) using the functions $\{\overline{a}_i\}_{i\leq K},
\overline{b},\{\overline{c}_i\}_{i\leq K}$ (thus $k_{v,\del}=k_v(\{a_{i,\del}\}_{i\leq K},
b_{\del},\{c_{i,\del}\}_{i\leq K})$).
We define
$$\rho_v(\{\overline{a}_i\}_{i\leq K},\overline{b},\{\overline{c}_i\}_{i\leq K})
:=\rho(k_v(\{\overline{a}_i\}_{i\leq K},\overline{b},\{\overline{c}_i\}_{i\leq K})
,\mu(\{\overline{a}_i\}_{i\leq K},\overline{b},\{\overline{c}_i\}_{i\leq K})).$$
We shall simply write $\rho_{v,\del}$ for $\rho_v(\{a_{i,\del}\}_{i\leq K},
b_{\del},\{c_{i,\del}\}_{i\leq K})$ and
$\rho_{v}$ for $\rho_v(\{a_{i}\}_{i\leq K},
b,\{c_{i}\}_{i\leq K})$.

For $1\leq i\leq K$, let
$$m_{a_i}=\min\{k\in\Z_{\geq 0}:\ \frac{d^k a_i}{dt^k}(0)\neq 0\}.$$
Since $a_i$ is real analytic around zero and $x_j(t)>0$ for $t>0$
and $j\in\Omega_K$ (Theorem 2.1 in \cite{spencer}), $m_{a_i}$ is necessarily finite.
Define $m_{c_i}$ for $1\leq i\leq K$ similarly. Let
$$\alpha:=\frac{1}{2}\min\{1+\frac{1}{m_{a_i}},1+\frac{2}{m_{c_i}}:\ 1\leq i\leq K\}.$$
The following Lemma is the analogue of Lemma \ref{lem:BF norm difference}.
\begin{lemma}\label{lem:BSR norm difference}
Let $\alpha$ be as above. Then for any $\alpha'\in(0,\alpha)$ we have
$$|\rho_{v,\del}-\rho_v|\leq C\del^{\alpha'},\text{ for every }v\in[0,T].$$
The constant $C$ depends only on $\alpha'$.
\end{lemma}
The proof of Lemma \ref{lem:BSR norm difference} will be given in
Section \ref{sec:Proof of Lemma for BSR}. As in the Bohman-Frieze process,
we have $\rho_{t_c}-\rho_v\geq C(t_c-v)$ for $v<t_c$ (Lemma 5.12 in \cite{bhamidi2}).
Hence, for any $\zeta'<\zeta:=\alpha/2$, we can choose $\gamma$ close to $1/2$
and $\alpha'$ close to $\alpha$ so that $\gamma\alpha'>\zeta'$.
Since $\rho_{t_c}=1$ (Theorem 1.3 in \cite{bhamidi2}), we conclude from
Lemma \ref{lem:BSR norm difference} that
\begin{align}\label{eqn:norm left derivative}
1-\rho_{t,\del}&\geq 1-(\rho_t+Cn^{-\gamma\alpha'})\\
&=(\rho_{t_c}-\rho_t)-Cn^{-\gamma\alpha'}\geq C'(t_c-t),\nonumber
\end{align}
the last inequality holds as $t\leq t_c-n^{-\zeta'}$.
\subsection{Proof of Theorem \ref{thm:continuous time BSR}}
Let $t=t(n)$ satisfy $t\leq t_c-n^{-\zeta'}$ where $\zeta'<\zeta:=\alpha/2$.
Let us choose $\gamma$ and $\alpha'$ as described above so that \eqref{eqn:norm left derivative} holds.

Let $C_1^{RG}(t)$ be the largest component of
$RG_{n,t} (=RG_{n,t}(\{a_{i,\del}\}_{i\leq K}, b_{\del}, \{c_{i,\del}\}_{i\leq K}))$ and let
$L_1^{RG}(t):=\sum_{x\in C_1^{RG}(t)}\phi_t(x)$ be the volume
of $C_1^{RG}(t)$. From Lemma 4.4 of \cite{bhamidi2}, it follows that
\begin{equation}\label{eqn:BSR1}
\PR(L_1^{BSR}(t)\geq m)\leq \PR(L_1^{RG}(t)\geq m)+C\exp(-C'n^{1-2\gamma})
\end{equation}
for $m\geq K+1$. For $x=(s,w)\in X$, define $I(x)=s$. Let $\poi_n$ be a Poisson
process in $X$ with rate $n\mu_{\del}$ and let $N_{n,t}:=\{x\in\poi_n:\ I(x)\leq t\}$.
Then for any $A>0$
\begin{equation}\label{eqn:BSR2}
\PR(L_1^{RG}(t)\geq m)\leq\PR(L_1^{RG}(t)\geq m,N_{n,t}\leq nA)+\PR(N_{n,t}>nA).
\end{equation}
Let $I_1\leq\hdots\leq I_{N_{n,t}}$ be an ordering of elements of
$\{I(x):\ x\in\poi_n, I(x)\leq t\}$. Then
\begin{align}\label{eqn:BSR3}
&\PR(L_1^{RG}(t)\geq m,N_{n,t}\leq nA)\\
&\hskip10pt=\sum_{j=1}^{nA}\PR(L_1^{RG}(t)\geq m,\ j\leq N_{n,t}\leq nA,\ I_j=\min\{I(x):\ x\in C_1^{RG}(t)\})\nonumber\\
&\hskip10pt\leq \sum_{j=1}^{nA}\PR\left(\left\{L_1^{RG}(t)\geq m,\ I_j=\min\{I(x):\ x\in C_1^{RG}(t)\}\right\}|\ E_j\right)\nonumber
\end{align}
where $E_j=\{N_{n,t}\geq j\}$. Define the event $F_j$ as follows,
$$F_j:=\{x\notin C^{RG}(I_j,t)\text{ whenever }I(x)<I_j\}$$
where $C^{RG}(I_j,t)$ is the component in $RG_{n,t}$ containing the point
$x_j$ such that $I(x_j)=I_j$. From \eqref{eqn:BSR3}, it follows that
\begin{equation}\label{eqn:BSR4}
\PR(L_1^{RG}(t)\geq m,N_{n,t}\leq nA)\leq\sum_{j=1}^{nA}
\PR\left(\left\{\mathrm{volume}(C^{RG}(I_j,t))\geq m,\ F_j\right\}|\ E_j\right).
\end{equation}

For any $w_0\in W$ with $w_0(0)=2K$, define a branching process on $[0,t]\times W$
starting from $x_0=(0,w_0)$ exactly as in the proof of Theorem \ref{thm:continuous time BF}
using the kernel $k_{t,\del}$ and the measure $\mu_{\del}$. Let $G(x_0)$ be as in
the proof of Theorem \ref{thm:continuous time BF}.
Consider the event $E_{w_0}:=\{x_0\in\poi_n\}$. On the event $E_{w_0}$, let $C^{RG}(x_0,t)$
denote the component of $x_0$ in $RG_{n,t}$.
Then, as
in the proof of Theorem \ref{thm:continuous time BF},
\begin{align}\label{eqn:BSR5}
&\PR\left(\left\{\mathrm{volume}(C^{RG}(I_j,t))\geq m,\ F_j\right\}|\ E_j\right)\\
&\hskip10pt\leq\int_{w_0\in W}\PR(\mathrm{volume}(C^{RG}(x_0,t))\geq m| E_{w_0})\nu_{K,0;\del}(dw_0)\nonumber\\
&\hskip10pt\leq\int_{w_0\in W}\PR(G(x_0)\geq m)\nu_{K,0;\del}(dw_0).\nonumber
\end{align}
It is easy to see that the analogues of Lemma \ref{lem:2} and Lemma \ref{lem:3}
remain true in the present setup with $\Delta=1-\rho_{t,\del}$.
Hence, we can proceed as before to conclude that for some positive constant
$C_5$ independent of $w_0$,
$$\E\exp(\eta\Delta^2 G(x_0))\leq\exp(C_5\eta w_0(T))$$
whenever $\eta\leq\eta_0$, an absolute constant free of $w_0$.
Hence
\begin{align}\label{eqn:BSR6}
\int_{w_0\in W}\PR(G(x_0)\geq m)\nu_{K,0;\del}(dw_0)
&\leq\exp(-\eta\Delta^2 m)\int_{W}\exp(C_5\eta w_0(T))\nu_{K,0;\del}(dw_0)\\
&\leq C_6\exp(-\eta\Delta^2 m)\nonumber
\end{align}
for $\eta$ small enough.
Since $\Delta\geq C(t_c-t)$, we conclude from \eqref{eqn:BSR1}, \eqref{eqn:BSR2},
\eqref{eqn:BSR4}, \eqref{eqn:BSR5} and \eqref{eqn:BSR6} that
$$\PR\left(L_1^{BSR}(t)\geq\frac{\overline{C}\log n}{(t_c-t)^2}\right)\stackrel{n\to\infty}{\longrightarrow}0$$
by choosing $A$ sufficiently large and choosing
a large constant $\overline{C}$ accordingly. This completes the proof of Theorem \ref{thm:continuous time BSR}.

The proof of Corollary \ref{cor:discrete time BSR} is similar to the proof of Corollary \ref{cor:discrete time BF},
so we omit it.
\subsection{Proof of Lemma \ref{lem:BSR norm difference}}\label{sec:Proof of Lemma for BSR}
We shall change the functions $b,\{a_i\}_{i\leq K},\{c_i\}_{i\leq K}$
one at a time to go from $\rho_v$ to $\rho_{v,\del}$. Thus
$$|\rho_{v,\del}-\rho_{v}|\leq D(b,\del)+\sum_{i=1}^K D(a_i,\del)+\sum_{i=1}^K D(c_i,\del)$$
where
\begin{align*}
D(b,\del)=|\rho_v(\{a_j\}_{j\leq K},b_{\del},\{c_j\}_{j\leq K})-\rho_v|,
\end{align*}
\begin{align*}
D(a_i,\del)=&|\rho_v(\{a_{1,\del},\hdots,a_{i,\del},a_{i+1},\hdots,a_K\},b_{\del},\{c_j\}_{j\leq K})\\
&\hskip35pt -\rho_v(\{a_{1,\del},\hdots,a_{i-1,\del},a_{i},\hdots,a_K\},b_{\del},\{c_j\}_{j\leq K})|\text{ and}\\
D(c_i,\del)=&|\rho_v(\{a_{j,\del}\}_{j\leq K},b_{\del},\{c_{1,\del},\hdots,c_{i,\del},c_{i+1},\hdots,c_K\})\\
&\hskip35pt -\rho_v(\{a_{j,\del}\}_{j\leq K},b_{\del},\{c_{1,\del},\hdots,c_{i-1,\del},c_{i},\hdots,c_K\})|.
\end{align*}
As in the proof of Lemma \ref{lem:BF norm difference} (see \eqref{eqn:T2 bound}), we have
$$D(b,\del)\leq C\del.$$

It is easy to see that if $K\geq 2$, then $m_{a_K}\geq 2$ and $m_{c_K}\geq 1$
(see the discussion in Section \ref{sec:examples}). Fix $i,j\leq K$ such that
$m_{a_i}\geq 2$ and $m_{c_{j}}\geq 1$. We shall show that
$D(a_i,\del)\leq C\del^{\frac{1}{2}+\frac{1}{2m_{a_i}}}$
and $D(c_{j},\del)\leq C(\beta)\del^{\beta}$ for every $\beta<1\wedge\left(\frac{1}{2}+\frac{1}{m_{c_{j}}}\right)$
and this will suffice. (It will follow from our proof that
if $m_{a_{p}}=0$ or $1$
then $D(a_p,\del)\leq C \del{|\log\del|}^{m_{a_p}/2}$. Similarly
if $m_{c_q}=0$, then $D(c_q,\del)\leq C\del$.)

Define the measures $\sigma^{(\ell)}$ for $\ell=1,2,3,4$ on $X$ as follows,
\begin{align*}
\sigma^{(\ell)}=
\left\{
\begin{array}{l}
\mu\left(\{a_{1,\del},\hdots,a_{i-1,\del},a_{i},\hdots,a_K\},b_{\del},\{c_p\}_{p\leq K}\right), \text{ if }\ell=1,\\
\mu\left(\{a_{1,\del},\hdots,a_{i,\del},a_{i+1},\hdots,a_K\},b_{\del},\{c_p\}_{p\leq K}\right),\text{ if }\ell=2,\\
\mu\left(\{a_{p,\del}\}_{p\leq K},b_{\del},\{c_{1,\del},\hdots,c_{j-1,\del},c_{j},\hdots,c_K\}\right), \text{ if }\ell=3,\\
\mu\left(\{a_{p,\del}\}_{p\leq K},b_{\del},\{c_{1,\del},\hdots,c_{j,\del},c_{j+1},\hdots,c_K\}\right), \text{ if }\ell=4.\\
\end{array}
\right.
\end{align*}
From Lemma 5.7 of \cite{bhamidi2}, it follows that $\sigma^{(1)}<<\sigma^{(2)}$
and $\sigma^{(3)}<<\sigma^{(4)}$. Further, for $(s,w)\in X$,
$$f_1(s,w):=\frac{d\sigma^{(1)}}{d\sigma^{(2)}}(s,w)=\I\{J(s)\neq K+i\}+\I\{J(s)=K+i\}\frac{a_i(s)}{a_{i,\del}(s)}$$
and
$$f_2(s,w):=\frac{d\sigma^{(3)}}{d\sigma^{(4)}}(s,w)
=\prod_{p:J(\gamma_p(s,w))=j}\frac{c_j(\gamma_p(s,w))}{c_{j,\del}(\gamma_p(s,w))}\cdot\exp\left(\int_s^T w(u)(c_{j,\del}(u)-c_j(u))du\right)
$$
where $J(u)=J(u;w):=w(u)-w(u-)$ and $\gamma_p(s,w)$ is the time of the $p$th jump of $w$
after time $s$. Since
$\rho\left(k_{v,\del},\sigma^{(1)}\right)=\rho\left(k_{v,\del}(x,y)\sqrt{f_1(x)f_1(y)},\sigma^{(2)}\right)$,
\begin{align*}
D(a_i,\del)&=|\rho\left(k_{v,\del},\sigma^{(2)}\right)
-\rho\left(k_{v,\del}(x,y)\sqrt{f_1(x)f_1(y)},\sigma^{(2)}\right)|\\
&\leq \left[\int_X\int_X k_{v,\del}^2(x,y)
\left(\sqrt{f_1(x)f_1(y)}-1\right)^2\sigma^{(2)}(dx)\sigma^{(2)}(dy)\right]^{\frac{1}{2}}.
\end{align*}
Writing $x=(s_1,w_1)$ and $y=(s_2,w_2)$, we have $k_{v,\del}(x,y)\leq Cw_1(v)w_2(v)$. Also, we have
$(\sqrt{f_1(x)f_1(y)}-1)^2\leq 2(\sqrt{f_1(x)}-1)^2+2(\sqrt{f_1(y)}-1)^2$. These inequalities and a little
work yield
\begin{align*}
D(a_i,\del)&\leq C\left(\int_0^v a_{i,\del}(s_1)
\left(1-\sqrt{\frac{a_i(s_1)}{a_{i,\del}(s_1)}}\right)^2ds_1\right)^{\frac{1}{2}}\\
&\leq C\left[\int_0^{\del^{\frac{1}{m_{a_i}}}}
\left(\sqrt{a_{i,\del}(s_1)}-\sqrt{a_i(s_1)}\right)^2ds_1
+\int_{\del^{\frac{1}{m_{a_i}}}}^v a_{i,\del}(s_1)
\left(1-\frac{a_i(s_1)}{a_{i,\del}(s_1)}\right)^2ds_1\right]^{\frac{1}{2}}\\
&\leq C\left[\int_0^{\del^{\frac{1}{m_{a_i}}}} (\sqrt{\del})^2\ ds_1
+\del^2\int_{\del^{\frac{1}{m_{a_i}}}}^v \frac{ds_1}{a_{i,\del}(s_1)}\right]^{\frac{1}{2}}.
\end{align*}
From the definition of $m_{a_i}$ and the fact that $x_j(\cdot)$ is
bounded away from zero on $[\eps,T]$ for each $j\in\Omega_K$ and $\eps>0$ (see \cite{spencer}),
it follows that $a_{i,\del}(s_1)\geq a_i(s_1)\geq Cs_1^{m_{a_i}}$
on $[0,T]$. Hence
\begin{align*}
D(a_i,\del)\leq C\left(\del^{1+\frac{1}{m_{a_i}}}+\del^2/(\del^{\frac{1}{m_{a_i}}})^{m_{a_i}-1}\right)^{\frac{1}{2}}
=C'\del^{\frac{1}{2}+\frac{1}{2m_{a_i}}}.
\end{align*}

 A similar argument will yield,
\begin{align}\label{eqn:BSR7}
D(c_j,\del)&\leq C\left[\int_X\int_X w_1^2(v)w_2^2(v)\left(\sqrt{f_2(x)f_2(y)}-1\right)^2\sigma^{(4)}(dx)\sigma^{(4)}(dy)\right]^{\frac{1}{2}}\\
&\leq C'\left[\int_X w_1^2(v)\left(\sqrt{f_2(x)}-1\right)^2\sigma^{(4)}(dx)\right]^{\frac{1}{2}}\nonumber
\end{align}
Let
$\tau_p=\gamma_p(s_1,w_1)$. From \eqref{eqn:BSR7} and an argument similar to the
ones used in the proof of Lemma \ref{lem:BF norm difference},
we get
\begin{equation}\label{eqn:bound on D(c_j)}
D(c_j,\del)\leq C(\del^2+T_1)^{\frac{1}{2}}
\end{equation}
where,
$$T_1=\int_X w_1^2(v)\left(1-\sqrt{\prod_{p: J(\tau_p)=j}
\frac{c_j(\tau_p)}{c_{j,\del}(\tau_p)}}\right)^2\sigma^{(4)}(dx).$$

Let us define the measures $\nu^{(4)}_{\ell,s;\del}$ for $\ell=1,\hdots,K$
on $W$ such that
$\sigma^{(4)}(d(s,w))=\sum_{\ell=1}^K a_{\ell,\del}(s)ds\ \nu_{\ell,s;\del}^{(4)}(dw)$.
Also define $S(w_1):=(\#\{s\leq T : w_1(s)-w_1(s-)>0\}-1)$. Then
$$\left(1-\sqrt{\prod_{p: J(\tau_p)=j}
\frac{c_j(\tau_p)}{c_{j,\del}(\tau_p)}}\right)^2
\leq w_1(T)\sum_{p=1}^{\infty}\left(1-
\sqrt{\frac{c_j(\tau_p)}{c_{j,\del}(\tau_p)}}\right)^2 \I\{S(w_1)\geq p, J(\tau_p)=j\}
$$
and hence
\begin{equation}\label{eqn:BSR8}
T_1\leq \sum_{\ell=1}^K\int_{s_1=0}^T a_{\ell,\del}(s_1) I_{\ell}\ ds_1
\end{equation}
where
\begin{align}\label{eqn:numberit}
I_{\ell}&=\sum_{p=1}^{\infty}\int_W w_1^2(v)w_1(T)
\left(1-\sqrt{\frac{c_j(\tau_p)}{c_{j,\del}(\tau_p)}}\right)^2 \I\{S(w_1)\geq p, J(\tau_p)=j\}\
\nu^{(4)}_{\ell,s_1;\del}(dw_1)\\
&\leq \sum_{p=1}^{\infty}\E_{\nu^{(4)}_{\ell,s_1;\del}}
\left[w_1^3(T)\I\{w_1(T)\geq p+\ell\}\cdot
\left(1-\sqrt{\frac{c_j(\tau_p)}{c_{j,\del}(\tau_p)}}\right)^2 \I\{S(w_1)\geq p, J(\tau_p)=j\}\right]\nonumber\\
&\leq \sum_{p=1}^{\infty}C\exp\left(-\frac{-C'p}{2q'}\right)
\left[\E_{\nu^{(4)}_{\ell,s_1;\del}}
\left[\left(1-\sqrt{\frac{c_j(\tau_p)}{c_{j,\del}(\tau_p)}}\right)^{2q}
 \I\{S(w_1)\geq p, J(\tau_p)=j\}\right]\right]^{\frac{1}{q}}.\nonumber
\end{align}
Here $q,q'>1$ satisfy $q^{-1}+q'^{-1}=1$, later we shall specify the appropriate value of $q$.
In the last inequality
we have used the fact that $w_1$ has an exponentially decaying
tail (Lemma 5.4 of \cite{bhamidi2}).

Note that the measure
$\xi(\cdot):=\nu^{(4)}_{\ell,s_1;\del}
\{S(w_1)\geq p, J(\tau_p)=j, \tau_p\in\cdot\}$ on
$(s_1,T]$ is absolutely continuous with respect to the Lebesgue measure
($\mathfrak{Leb}$) on $(s_1,T]$ and a routine computation will show that
$\left|\frac{d\xi}{d\mathfrak{Leb}}(u)\right|\leq K(p+1)c_{j,\del}(u)$.

Let us choose $\beta\in(0,1]$ and $q>1$ so that $1<(2q\beta-1)m_{c_j}<2$.
From \eqref{eqn:numberit} and the fact that $c_{j}(u)\geq C u^{m_{c_j}}$ on $[0,T]$, we get
\begin{align*}
I_{\ell}&\leq\left(\sum_{p=1}^{\infty}C\exp\left(-\frac{-C'p}{2q'}\right)(K(p+1))^{\frac{1}{q}}\right)
\left(\int_{s_1}^T\left(1-\sqrt{\frac{c_j(u)}{c_{j,\del}(u)}}\right)^{2q}c_{j,\del}(u)\ du\right)^{\frac{1}{q}}\\
&\leq C''\left(\int_{s_1}^T\left(1-\frac{c_j(u)}{c_{j,\del}(u)}\right)^{2q\beta}c_{j,\del}(u)\ du\right)^{\frac{1}{q}}\\
&\leq C'''\del^{2\beta}\left(\int_{s_1}^T \frac{du}{u^{(2q\beta-1)m_{c_j}}}\right)^{\frac{1}{q}}\\
&\leq C''''\del^{2\beta}\cdot s_1^{(1-(2q\beta-1)m_{c_j})/q}
\leq C'''''\del^{2\beta}\cdot s_1^{1-(2q\beta-1)m_{c_j}}.
\end{align*}
From \eqref{eqn:BSR8}, we conclude that
$$T_1\leq C\del^{2\beta}\int_{0}^T \frac{ds_1}{s_1^{(2q\beta-1)m_{c_j}-1}}\leq C'\del^{2\beta},$$
the constant $C'$ is finite by the choice of $q$ and $\beta$. Since we can choose
$\beta$ arbitrarily close to $1\wedge\left(\frac{1}{2}+\frac{1}{m_{c_j}}\right)$,
the last inequality together with \eqref{eqn:bound on D(c_j)} yields the desired bound.
\subsection{Some examples}\label{sec:examples}
Let $S=\{\{1,1\},\{1,2\}\}$. Consider a bounded size rule
with $K=2$ which satisfies either\\
(I) $(2,2,\alpha,\beta)\in F$ for some $\alpha,\beta$ such that $\{\alpha,\beta\}\in S$ or\\
(II) $(\alpha',\beta',2,2)\in F^c$ for some $\alpha',\beta'$ such that $\{\alpha',\beta'\}\in S$.

It is easy to check from the system of ODEs \eqref{eqn:ODE} that $x_2'(0)=1$.
Hence $x_2(u)\approx u$ around zero. Also $\Delta((1,2,1,2),\omega)=3/2$.
Hence $x_{\omega}'(u)\geq \frac{3}{2} x_1^2(u)x_2^2(u)\geq C u^2$, which implies
that $x_{\omega}(u)\geq C'u^3$ on $[0,T]$.

By considering the tuple $(1,2,1,2)$, we get
\begin{align*}
a_1(u)\geq F_{1,2}(x_1(u),x_2(u),x_{\omega}(u))\geq \frac{1}{2}x_1^2(u)x_2^2(u)\geq Cu^2,
\end{align*}
hence $m_{a_1}\leq 2$. Under the assumptions on $F$, either
$a_2(u)=F_{2,2}(x_1(u),x_2(u),x_{\omega}(u))\geq \frac{1}{2}x_2^2(u)x_{\alpha}(u)x_{\beta}(u)$
for some $\{\alpha,\beta\}\in S$ or
$a_2(u)\geq \frac{1}{2}x_2^2(u)x_{\alpha'}(u)x_{\beta'}(u)$
for some $\{\alpha',\beta'\}\in S$.
In either case $m_{a_2}=2\text{ or }3$.
Considering the tuple $(1,\omega,1,\omega)$, we get
$$c_1(u)=F_{1,\omega}(x_1(u),x_2(u),x_{\omega}(u))/x_{\omega}(u)\geq \frac{1}{2}x_1^2(u)x_{\omega}(u)\geq Cu^3.$$
Hence $m_{c_1}\leq 3$. Finally $c_2(u)\geq \frac{1}{2}x_2^2(u)x_{\omega}(u)\geq Cu^5$,
by considering the tuple $(2,\omega,2,\omega)$. Hence $m_{c_2}\leq 5$ (note also that
$m_{c_2}\geq 1$ since $x_2(u)$ is a factor of $c_2(u)$). From these
inequalities, we see that $\zeta(F)\geq 1/3$, hence the bound on $L_1^{BSR}(t)$
and $L_1^{DBSR}(t)$ from Theorem \ref{thm:continuous time BSR} and Corollary \ref{cor:discrete time BSR}
hold all the way to the critical window.

When $K\geq 3$, $x_{K}'(0)=0$. Hence $x_K(u)\leq Cu^2$ around zero.
Since $a_K(u)=x_K^2(u)f(u)$ for some bounded function $f$, $m_{a_K}\geq 4$.
Hence $\zeta(F)\leq 1/4\cdot(1+1/m_{a_K})\leq 1/4\cdot(1+1/4)<1/3$.  So our proof does not guarantee
that the upper bound holds up to the critical window.

\section*{Acknowledgements}
I thank Joel Spencer for suggesting
this problem to me and for many helpful comments and lively discussions.
I also thank Shankar Bhamidi, Amarjit Budhiraja
and Xuan Wang for carefully reading the first draft of this paper
and for many useful comments and suggestions, particularly
for suggesting an improvement in an earlier version of
Lemma \ref{lem:BF norm difference}.

\end{document}